\documentclass[11pt,twoside]{amsart}
\usepackage[all]{xy}
        \usepackage {amssymb,latexsym,amsthm,amsmath,mathtools}
        \usepackage{enumerate,color}
        
\usepackage[pagebackref=true, colorlinks, linkcolor=blue, citecolor=blue]{hyperref}
\usepackage{mathtools}
 \usepackage{cleveref}

\long\def\symbolfootnote[#1]#2{\begingroup%
\def\thefootnote{\fnsymbol{footnote}}\footnote[#1]{#2}\endgroup}

\makeatletter
\def\imod#1{\allowbreak\mkern10mu({\operator@font mod}\,\,#1)}
\makeatother

\newcommand{\qq}{\mathbb Q}

\newcommand{\nat}{\mathbb{N}} 
\newcommand{\zz}{\mathbb Z}

\newcommand{\mg}[1]{{#1}^{\times}}
\newcommand{\sq}[1]{{#1}^{\times 2}}

\newcommand{\ovl}{\overline}

\newcommand{\car}{\mathrm{char}}

\newcommand{\trdeg}{\mathrm{tr.deg}}

\newtheorem{theorem}{Theorem}[section]
\newtheorem{lemma}[theorem]{Lemma}

\newtheorem{proposition}[theorem]{Proposition}
\newtheorem{question}[theorem]{Question}
\newtheorem*{theorem*}{Theorem}
\theoremstyle{definition}

\newtheorem{remark}[theorem]{Remark}
\newtheorem{example}[theorem]{Example}
\numberwithin{equation}{section}

\newcommand{\ignore}[1]{}

\newcommand{\mynote}[1]{}

\begin{document}
\setcounter{section}{0}
\title{Splitting of differential Quaternion algebras}
\author{Parul Gupta}
\email{parulgupta1211@gmail.com}
\author{Yashpreet Kaur}
\email{yashpreetkm@gmail.com }
\author{Anupam Singh}
\email{anupamk18@gmail.com}
\address{IISER Pune, Dr. Homi Bhabha Road, Pashan, Pune 411 008, India}
\thanks{The second named author is supported by NBHM (DAE, Govt. of India): 0204/16(33)/2020/R$\&$D-II/26. The third named author is funded by an NBHM research project grant and partly by SERB CRG/2019/000271 for this research.}
\today


\begin{abstract}
We study differential splitting fields of quaternion algebras with derivations. A quaternion algebra over a field $k$ is always split by a quadratic extension of $k$. However, a differential quaternion algebra need not be split over any algebraic extension of $k$. We use solutions of certain Riccati equations to provide bounds on the transcendence degree of splitting fields of a differential quaternion algebra.  
\end{abstract}

\maketitle

\section{Introduction}
In this paper, we study differential splitting fields of quaternion algebras with derivations. Quaternion algebras are 4-dimensional central simple algebras. Derivations on central simple algebras are studied by Hochschild \cite{Ho}, Hoechsmann\cite{Hoe}, Juan and Magid \cite{LJARM}, Kulshrestha and Srinivasan \cite{AKVRS}, Gupta, Kaur and Singh \cite{GKS}.

A notion of splitting fields for a differential central simple algebra was introduced in \cite{LJARM}. Authors also proved the existence of a finitely generated differential field extensions that split the given differential algebra.
It is clear from  \cite[Proposition 2]{LJARM} that every $m^2$-dimensional differential central simple algebra is split by a differential field extension of transcendence degree at most $m^2$. In particular, every differential quaternion algebra is split by a differential field extension of transcendence degree at most $4$ (also see \cite[Theorem 4.1]{AKVRS}).  This raises a question:
\begin{question}
 What is the optimal bound  on transcendence degree of a splitting field?
\end{question}

We answer the above question for quaternion algebras over a field $k$ with $\car(k)\neq2$. In \Cref{splitequivalence}, we provide a construction of differential splitting fields of a  differential quaternion algebra using differential fields containing solutions of a certain Riccati equation.  We then use this proposition to show the transcendence degree of such a splitting field is at most $3$ (see \Cref{splitequivalencec}). Moreover, for $0\leq r\leq 3$, we give examples of  differential quaternion algebras whose splitting fields have transcendence degree at least $r$ over $k$ (see \Cref{transdegthree} and \Cref{transdegone}). Thus, the bound $3$ is optimal.
In \Cref{finitefields}, we consider a particular class of differential quaternion algebras that are split over a transcendence degree at most one extension of $k$.

Let $k$ be a field containing a primitive $m$th root of unity, say $\omega$.
An $m^2$-dimensional central simple algebra over $k$ is a symbol algebra if it is generated by two elements $u,v$ such that $u^m, v^m \in \mg k$ and $vu = \omega uv$.  
For generators $u,v$ of a symbol algebra $A$, there is a derivation on $A$ that restricts to a given derivation on $k$ such that  $d(k(u))\subseteq k(u)$ and $d(k(v))\subseteq k(v)$; such a derivation is named as standard derivation.
It was shown in \cite[Theorem 5.4]{GKS} that a $m^2$-dimensional symbol algebra with standard derivation is split by a finite field extension of degree at most $m^2$ if $m$ is odd and $2m^2$ if $m$ is even. 
It is natural to ask the following questions.

\begin{question} \label{standardquestion}
	\begin{enumerate}[(i)]
		\item Can one classify all differential symbol algebras over a differential field that are split by finite field extensions?
		\item  Can one classify  all standard derivations on a given symbol algebra?
	\end{enumerate}
\end{question}

If $\car(k) \neq 2$ then any quaternion algebra is a symbol algebra. Thus, in particular, a quaternion algebra with standard derivation is split by a finite field extension of degree at most $8$ (also see \cite[Theorem 4.1]{AKVRS}). 
We show that its converse is not true (see \Cref{examplealgebraic})
In the case of quaternion algebras, we show that the above questions are inter-related.
 In \Cref{standardequivalance}, we prove that  a differential quaternion algebra is split by a finite differential field extension if and only if the derivation on the differential split algebra is standard. We hope to extend these results to an arbitrary symbol algebra in future.

\subsection*{Notation} Throughout this paper we fix  a field $k$ with $\car(k)\neq2$, an algebraic closure  $\ovl{k}$ of $k$.
Let $\mg k$ denote the multiplicative group of $k$.
For  $\alpha, \beta \in \mg k$, we  fix the quaternion algebra $Q = (\alpha, \beta)_{k}$ with generators $u,v$ such that $u^2 =\alpha, v^2 =\beta$ and $vu = - uv$.

\section{Differential splitting of central simple algebras}\label{dqa}

A ring (resp. algebra or field) $R$ together with an additive map $d\colon R \rightarrow R$ satisfying  $d(xy)=xd(y) + d(x)y$ for all $x, y \in R$, is called differential ring (resp. algebra or field) and is denoted by $(R, d)$. The map $d$ on $R$ is called a \textbf{derivation on $R$}. The set $C_{(R,d)} = \{ x\in R \mid d(x) =0\}$ is called the set of \textbf{constants of $(R,d)$}.

For any element $\vartheta \in R$, 
there is an \textbf{inner derivation} $\partial_\vartheta $ on $R$ given by $\partial_\vartheta (x)= x\vartheta -\vartheta  x$ for $x\in R$. Note that  $C_{(R,\partial_\vartheta)}$ is equal to the centralizer of $\vartheta$ in $R$.

A {\bf central simple algebra over $k$} is a finite dimensional $k$-algebra with center $k$ and no non-trivial two-sided ideals. 
 A central simple $k$-algebra $A$ \textbf{splits over a field extension $L/k$} if $A\otimes_k L\simeq M_m(L)$ as $L$-algebras and $L$ is called a \textbf{splitting field} of $A$. 

Let $(k,\,')$ be a differential field. For a field extension $L/k$, there is a derivation on $L$ that restricts to $'$ on $k$ (see \cite[Section 3]{Ros}). Moreover, when $L/k$ is a finite field extension, there is a unique extension of $'$ on $L$.  For convenience, we denote the derivations on field extensions by $'$ whenever the derivation is clear from the context. 

A central simple $k$-algebra $A$ together with a derivation $d$ that restricts to $'$ on k is called differential central simple algebra over $k$, and is denoted by $(A,d)$. The set of all derivations on $A$ that are extensions of $'$ is denoted by $Der(A/(k,\,'))$.

We say that $(A_1,d_1)$ and $(A_2,d_2)$ are isomorphic as differential $k$-algebras if there exists a  $k$-algebra isomorphism $\phi: (A_1,d_1) \rightarrow (A_2,d_2)$ such that $\phi \circ d_1 = d_2 \circ \phi$; and $\phi$ is called {\bf differential isomorphism}.

	The derivation $'$ on the matrix algebra  $M_m(k)$ defined by $(a_{ij})' = (a'_{ij})$ is 	called the \textbf{coordinate-wise derivation on $M_m(k)$}. For the matrix algebra $M_m(k)$, we have $ Der(M_m(k)/(k, \,')) = \{d_P =\,' + \partial_P \mid\ P\in M_m(k)\}$ (see \cite[Theorem 2]{Ami}). 	
	
	In \cite{LJARM}, the authors introduced the notion of differential splitting fields.
	 A differential central simple algebra $(A,d)$ over $(k,\, ')$ \textbf{splits over a differential field $(L,\, ')\supseteq (k,\, ')$} if $(A\otimes_k L, d^{\ast} ) \simeq (M_m(L),\, ')$ as differential $L$-algebras, where $d^{\ast} := d\otimes \, '=d\otimes id_L + id_A\otimes\,'$. In \cite{LJARM}, the authors proved the existence of differential splitting fields using the theory of Picard-Vessiot extensions. The following characterization of isomorphic differential matrix algebras was crucial in constructing differential splitting fields.

\begin{proposition}\cite[Proposition 2]{LJARM}\label{JM_split}
	Let $P\in M_m(k) \setminus \{0\}$.  Then $(M_m(k), d_P)$ and $(M_m(k), \,')$ are  isomorphic as differential $k$-algebras if and only if there exists $F \in GL_m(k)$ such that $F' =PF$. 
\end{proposition}

For a quaternion algebra $Q$ over $k$,
 a derivation $d\in Der(Q/(k,\,'))$ is called a {\bf standard derivation} if there exists $u,v\in Q^0$ such that $u^2, v^2 \in \mg k,\  v u=- u v,\ d(u)\in k(u)$ and $d(v)\in k(v)$, and we denote this derivation by $d_{(u,v)}$. 
In \cite[Proposition 3.2]{AKVRS}, it was shown that for every derivation  $d\in Der(Q/(k,\, '))$, there exists  a unique $\vartheta\in Q^0$ such that $d = d_{(u,v)}+\partial_\vartheta$. Furthermore, this characterization was used to construct splitting fields of quaternion algebras with derivations. 
We recall results related to differential splitting of quaternion algebras.

\begin{theorem}\label{Quatex}
Let $(k,\, ')$ be a differential field  with $\car(k)=0$. 
Let $Q =(\alpha, \beta)_k$ be a quaternion algebra.

\begin{enumerate}

\item \cite[Theorem 4.1]{AKVRS} \label{Quatex2} The differential algebra $(Q,d)$ is split by a finitely 
generated differential field extension  of $(k,\, ')$. The algebra $(Q,d_{(u,v)})$ is split by a finite extension of $(k,\,')$.

\item \cite[Theorem 4.5]{AKVRS} \label{Quatex3} If $\,'$ is the zero derivation and $Q$ is a division algebra with non-zero derivation $d$, then $(Q,d)$ is split by a field extension of transcendence degree $1$ and is not split by any algebraic extension of $k$.   
\end{enumerate}
\end{theorem}
 
We are now interested in constructing differential splitting fields for an arbitrary differential quaternion algebra $(Q,d)$.

\section{Solutions of Riccati equations}{}
In the process of finding the splitting fields of a differential quaternion algebra, we come across a certain types of first order differential equations with variable coefficients. 
In this section, we will note some essential results concerning  homogeneous linear differential equations and Riccati equations.

A first order differential equation over $k$ of the form
$$X' = \alpha_0 + \alpha_1 X+\alpha_2 X^2,$$
where $\alpha_0,\alpha_2\neq0$, is called {\bf Riccati equation}. If  $\alpha_0=\alpha_2=0$ then equation becomes a  homogenous linear differential equation. In general, solutions of a Riccati equation are not known, but there are methods to solve particular cases. We will use Riccati equations whose solutions are known to construct examples of  differential quaternion algebras that split over fields with fixed transcendence degrees. For this purpose, in the following results, we relate solutions of Riccati equations to that of a certain linear differential equation.

\begin{proposition}\label{constantalg}Let $a\in k$ and let $\theta\in \ovl k$ be a solution of  $Y'=aY$ such that  $n:=[k[\theta]:k]$ is minimal. Then $\theta^n\in k$. 
\end{proposition}

\begin{proof}
	Consider the polynomial ring $k[X]$ with derivation $D$ such that $D(X)=aX$.
	Let  $p(X)=X^n+a_{n-1}X^{n-1}+\dots+a_1X+a_0$ be the minimal polynomial of $\theta$ over $k$. 
	Consider the differential homomorphism $\phi:k[X]\to k[\theta]$ defined by $\phi(X)=\theta$. Then the kernel of $\phi$ is the prime ideal generated by $p(X)$. Since $D(p)(\theta)=0$, $p$ divides $D(p)$ and whereby $D(p)=nap$. On comparing the coefficients, we obtain $D(a_i)=(n-i)aa_i$ for  $0\leq i\leq n-1$. This implies $\sqrt[n-i]{a_i}$ is a solution of $Y'=aY$. Since $n$ is the minimal degree for such a solution, we obtain $a_i=0$ for $1\leq i\leq n-1$. Therefore, $p(X)=X^n+a_0$.	
\end{proof}

\begin{lemma}\label{zeropole}
Let $k(t)$ be the rational function field with derivation $t' = \alpha_0 + \alpha_1 t+\alpha_2 t^2$ where $\alpha_0, \alpha_1,\alpha_2 \in k$ and $\alpha_2\neq 0$.
Let $n\in \nat$ and $f \in k(t)$ be a solution of the differential equation $Y' = n\left( \frac{\alpha_1}{2} +\alpha_2 t\right)Y$.
Then each zero and pole of $f$ is a solution of $X' = \alpha_0 + \alpha_1 X+\alpha_2 X^2$.
\end{lemma}

\begin{proof}
Let $p =   \frac{\alpha_1}{2} +\alpha_2 t$. 
Since $p \notin k$, the solution $f \notin k$.
Let $\gamma_1, \ldots, \gamma_r \in \ovl{k}$, $\eta \in k$ and $m_1, \ldots, m_r \in \zz$ be such that $f = \eta\prod_{i=1}^r (t - \gamma_i)^{m_i} $.
Then 
$$f' =\eta \prod_{i=1}^r (t - \gamma_i)^{m_i-1} \cdot \left(\sum_{i=1}^{r} \Big((t'-\gamma_i')\prod_{j\neq i} (t - \gamma_j)\Big)\right) +\frac{\eta'}{\eta} f.$$ 
Since $f'=npf$, for $1\leq i\leq r$, we have $(t-\gamma_i)$ divides $(t'- \gamma_i') = \alpha_0 + \alpha_1 t +\alpha_2 t^2 -\gamma_i' $,
which implies that each $\gamma_i$ is a solution of $X' = \alpha_0 + \alpha_1 X+\alpha_2 X^2$.
\end{proof}

\begin{lemma}\label{noalgsolricatti}
Assume that the non-linear differential equation 
\begin{equation}\label{ricattieq}
 X' = \alpha_0 + \alpha_1 X +\alpha_2 X^2
\end{equation}
over $k$ has no algebraic solution and $\lambda_1,\lambda_2$ be its distinct transcendental solutions. 
Then the differential equation $Y' = \left( \frac{\alpha_1}{2} +\alpha_2 \lambda_1\right)Y$ has no algebraic solution over $k(\lambda_1,\lambda_2)$.
\end{lemma}

\begin{proof}
Set $L = k(\lambda_1,\lambda_2)$ and $p = \frac{\alpha_1}{2} +\alpha_2 \lambda_1$.
Suppose that $Y' = p Y$ has a solution in $\ovl{L}$.

First assume that $\lambda_2 \in \ovl{k(\lambda_1)}$. 
Then $\ovl{L} = \ovl{k(\lambda_1)}$. 
Let $\theta \in \ovl{L}$ be a solution of $Y' = p Y$ such that $n = [k(\lambda_1,\theta): k(\lambda_1)]$ is minimal.
By \Cref{constantalg}, it follows that $\theta^n = f \in k(\lambda_1)$. 
Then $f' = npf $.
Since \Cref{ricattieq} has no solution in $\ovl{k}$, by \Cref{zeropole} $f \in k$.
This implies $p \in k$, which is a contradiction. 

Now assume that $\lambda_2 \notin \ovl{k(\lambda_1)}$. 
Set $K =k(\lambda_2)$.
Note that $\lambda_2$ is the only solution of \Cref{ricattieq} in $\ovl{K}$.
Let $\theta \in \ovl{L}$ be a solution of $Y' = p Y$ such that $n = [K(\lambda_1,\theta): K(\lambda_1)]$ is minimal.
Then, by \Cref{constantalg},  $\theta^n = f \in K(\lambda_1)$ and $f' = npf$.
By \Cref{zeropole}, we have that $f  = \eta (\lambda_1-\lambda_2)^m$ for some $m\in \zz$ and $\eta \in K$.

Substituting $f= \eta (\lambda_1 -\lambda_2)^{m} $ in  $f' = npf $, we further get 
$$\frac{\eta'}{n\eta} + \frac{m}{n} \big(\alpha_1 + \alpha_2 (\lambda_1- \lambda_2)\big) = \frac{\alpha_1}{2} +\alpha_2 \lambda_1 \,. $$
On comparing the coefficients, we obtain  $m = n $ and $(\eta^{-1})' = n\left( \frac{\alpha_1}{2}+ \alpha_2 \lambda_2\right)\eta^{-1}$.
Since \Cref{ricattieq} has no solution in $\ovl{k}$, using \Cref{zeropole}, we arrive at a contradiction as above .
\end{proof}

\section{Splitting of differential quaternion algebras}{}

	Let	$Q=(\alpha,\beta)_k$ be a quaternion algebra with generators $u,v$.
	For $\xi \in \ovl{k}$ such that $\xi^2 = \alpha$, the homomorphism $\mathbf{\Phi_{Q, \xi}}: Q\otimes_kk(\xi) \rightarrow M_2(k(\xi)) $
	determined by
	\begin{center}
		$u \otimes 1 \mapsto A =  \left(
		\begin{matrix}
			\xi	& \\
			&-\xi\\
		\end{matrix} \right) ;\,\,\,\,
		v \otimes 1 \mapsto  B = \left(
		\begin{matrix}
			& \beta \\
			1 &\\
		\end{matrix} \right),$
	\end{center}
	and $\lambda\otimes r \mapsto \lambda r I_2$ where $\lambda \in k$ and $r\in k(\xi)$ is an isomorphism of $k(\xi)$-algebras. Let $d=d_{(u,v)}+\partial_{\vartheta}$, where $\vartheta=a_1u+a_2v+a_3uv\in Q^0$, be a derivation in $Der(Q/(k,\,'))$. The isomorphism $\Phi_{Q, \xi}$ induces a differential isomorphism $(Q\otimes_kk(\xi),d^\ast) \rightarrow (M_2(k(\xi)),d_P)$ where 
	\begin{equation}\label{matrixp}
		P=\left(
		\begin{matrix}
			a_1\xi+\frac{\beta'}{4\beta}	& (a_2+a_3\xi)\beta \\
			a_2-a_3\xi & 	-a_1\xi-\frac{\beta'}{4\beta}\\
		\end{matrix} \right),
	\end{equation}
	see \cite[Theorem 4.1]{AKVRS}.
	By  \Cref{JM_split}, the differential algebra $(M_2 (k(\xi)), d_P)$ is split over $(L, \, ')$ if and only if there exists $F\in GL_2(L)$ such that
	$F' = PF$. Write $ F = (f_{ij})$ with $f_{ij} \in L$. By comparing the matrix entries in the equation $F' = PF$ we obtain
	\begin{align}
		f_{11}' &= \left (a_1\xi + \frac{\beta'}{4\beta}\right) f_{11}+ (a_2+a_3\xi) \beta f_{21} \label{f11},\\
		f_{21}' &=  (a_2 - a_3\xi) f_{11} -\left (a_1\xi + \frac{\beta'}{4\beta}\right) f_{21},\label{f21}\\
		f_{12}' &= \left (a_1\xi + \frac{\beta'}{4\beta}\right) f_{12}+ (a_2+a_3\xi) \beta f_{22},\label{f12}\\
		f_{22}' &=  (a_2 - a_3\xi) f_{12} -\left (a_1\xi + \frac{\beta'}{4\beta}\right) f_{22}.\label{f22}
	\end{align}	
	Taking $L=k(\xi)(f_{11},f_{12},f_{21},f_{22})$, it is clear that $(Q,d)$ is always split by a differential field extension of transcendence degree at most $4$.

Here, we will show that a differential quaternion algebra is split by a differential field extension of transcendence degree at most $3$.

\begin{proposition}\label{splitequivalence}
Let $\xi \in \ovl{k}$ be such that $\xi^2 = \alpha$.
Let $(L,\,')$ be a differential field extension of $(k (\xi), \,')$. 
\begin{enumerate}[$(a)$]
\item\label{splitequivalencea} If $(Q,d)$ is split by $(L,\,')$, then the differential equation 
\begin{equation}\label{nonlineareq}
X' = (a_2+ a_3\xi)\beta + 2 \left(a_1\xi + \frac{\beta'}{4\beta}\right) X - (a_2- a_3\xi)X^2
\end{equation}
has at least two distinct solutions in $L$.  

\item\label{splitequivalenceb} If \Cref{nonlineareq}
has at least two distinct solutions (say $\lambda_1,\lambda_2$) in $L$, then 
$(Q,d)$ is split by a simple extension $(L(\mu),\,')$ where $\mu$ is a solution of the differential equation 
\begin{equation}\label{linearriccati}
\frac{Y'}{Y} =(a_2-a_3\xi)\lambda_1 -\left( a_1\xi+ \frac{\beta'}{4\beta}\right).	
\end{equation} 
\end{enumerate} 
\end{proposition}
\begin{proof}

\begin{enumerate}[$(a)$]
\item Assume that $(Q,d)$ is split by $(L,\,')$. 
Then the algebras $(Q\otimes k(\xi), d^\ast) \simeq (M_2(k(\xi)), d_P))$ are also split by $(L,\,')$. From \Cref{JM_split}, there exists $F=(f_{ij})\in GL_2(L)$ such that $F'=PF$. 
First suppose  that $a_2 =a_3 =0$. From Equations (\ref{f11}) to (\ref{f22}), we get that $f_{11},f_{12},f_{21}^{-1}$ and $f_{22}^{-1}$ are solutions of the  \Cref{nonlineareq}, $X' = \left (a_1\xi + \frac{\beta'}{4\beta}\right) X$.  
 
Now suppose $a_2 \neq 0$ or $a_3 \neq 0$. Since $F\in GL_2(L)$, from Equations (\ref{f11}) to (\ref{f22}), we conclude that $f_{ij} \neq 0$ for all $i,j$.
Considering the equations $\frac{(\ref{f11})}{f_{11}} - \frac{(\ref{f21})}{f_{21}}$ and $\frac{(\ref{f12})}{f_{12}} - \frac{(\ref{f22})}{f_{22}}$, we obtain that $\frac{f_{11}}{f_{21}}$ and $\frac{f_{12}}{ f_{22}}$ are solutions of \Cref{nonlineareq}.
Since $F\in GL_2(L)$, we have that $\frac{f_{11}}{f_{21}} \neq \frac{f_{12}}{ f_{22}}$.

\item 
Set 
$$F = \left(
\begin{matrix}
\lambda_1 \mu & \frac{\lambda_2}{\mu(\lambda_1-\lambda_2)} \\
\mu & \frac{1}{\mu(\lambda_1-\lambda_2)} \\
\end{matrix}\right).$$

Note that $\det(F)=1$ and $F'=PF$. Thus, by \Cref{JM_split} $(M_2(L(\mu)),d_P)\simeq(M_2(L(\mu)),\,')$. Hence, $(Q,d)$ splits over $(L(\mu),\,')$. 
\end{enumerate}
\end{proof}

Using this proposition, we can easily construct an example where differential splitting field of a quaternion algebra is of transcendence degree at most $1$. 
\begin{example}\label{splittransone}
	For $a\in k$, let  $d=d_{(u,v)}+\partial_{av}$ be a derivation on the quaternion algebra $Q=(\alpha,\beta)_k$. Then the \Cref{nonlineareq} becomes $X'=a(\beta-X^2)+\frac{\beta'}{2\beta}X$. Let $\eta\in \ovl{k}$ such that $\eta^2=\beta$, then $\eta,-\eta$ are  two distinct solutions of $X'=a(\beta-X^2)+\frac{\beta'}{2\beta}X$. From \Cref{splitequivalence}(\ref{splitequivalenceb}), $(Q,d)$ is split by $k(\eta,\mu)$ where $\mu'=(a\eta-\frac{\beta'}{4\beta}) \mu$. Note that the transcendence degree of $k(\eta,\mu)/k$ is at most $1$.

Since the generators $u,v$ and $uv$ are symmetric in the sense that $d_{(u,v)}=d_{(v,u)}=d_{(u,uv)}$, we have that $(Q,d)$, where $d=d_{(u,v)}+\partial_w$, is split by a differential field of transcendence degree at most $1$ if $w$ lies in $k(u)$, $k(v)$ or $k(uv)$. This generalizes \cite[Example 4.2]{AKVRS}.
\end{example}
Now we prove the main theorem of this section.
\begin{theorem}\label{splitequivalencec} $(Q,d)$ is split by a differential field extension of $(k,\,')$ of transcendence degree at~most $3$.
\end{theorem}
\begin{proof}
It follows from \Cref{splitequivalence}(\ref{splitequivalenceb}) that $(Q,d)$ is split by the differential field extension $L:=k(\xi,\lambda_1,\lambda_2,\mu)$ where $\lambda_1,\lambda_2$ are solutions of \Cref{nonlineareq} and $\mu$ is a solution of \Cref{linearriccati}. Clearly, $\trdeg_k(L)\leq 3$.
\end{proof}

\section{Examples of splitting fields with transcendence degree $2$ and $3$}
In this section, we construct examples of differential splitting fields with transcendence degree at least $2$ and $3$. We first observe the following result, which is crucial in the construction of such splitting fields.
\begin{proposition}\label{splittransnotone}
The differential quaternion algebra $(Q,d)$ is split by a field extension of transcendence degree at least $2$ if 
 the non-linear differential equation  (\ref{nonlineareq})
\begin{equation*}
X' = (a_2+ a_3\xi)\beta + 2 \left(a_1\xi + \frac{\beta'}{4\beta}\right) X - (a_2- a_3\xi)X^2
\end{equation*}
has no algebraic solutions over $k$.
Moreover, if any two solutions of \Cref{nonlineareq} are algebraically independent over $k$, then $(Q,d)$ is split by a field extension of transcendence degree at least $3$.

\end{proposition}

\begin{proof}
Let $(L,\,')$ be a differential field extension of $(k(\xi),\,')$ containing two distinct roots of \Cref{nonlineareq}.
Let $\lambda \in L$ be a solution of \Cref{nonlineareq}. 
By \Cref{splitequivalence}(\ref{splitequivalenceb}), $(Q,d)$ is split by an extension $L(\mu)$, where 
$\mu' =  (a_2-a_3\xi)\lambda -\left( a_1\xi+ \frac{\beta'}{4\beta}\right) \mu$.
Since \Cref{nonlineareq} has no algebraic solution over $k$, by \Cref{noalgsolricatti}, it follows that $\trdeg_k{L(\mu)} =\trdeg_k{L}+1$. 
\end{proof}

\begin{example}\label{transdegthree}
Let $(k,\,')$ be the differential field where $k =\qq(t)$ and $t' = 1$.
Let $Q=(1,t)_k$, $u,v\in Q^0$ such that $u^2=1$, $v^2=t$ and $vu=-uv$.

{\it (i)} Consider the derivation $d=d_{(u,v)}+ \partial_\vartheta$ on $Q$ with $\vartheta = \frac{1}{4t}(-u-2v+2uv)$.  
Then the \Cref{nonlineareq} becomes $X'=\frac{1}{t}X^2$. The general solution of this differential equation is given by $\frac{-1}{\log(t)+c}$, where $c$ is an arbitrary constant. This implies, $X'=\frac{1}{t}X^2$ has no solution in $\ovl k$ and infinitely many solutions in  $k(\log(t))$. Since the transcendence degree of $k(\log(t))$ over $k$ is $1$, from \Cref{splittransnotone}  we have the transcendence degree of splitting fields of $(Q,d)$ is at least $2$.

{\it (ii)} Consider the derivation  $d=d_{(u,v)}+\partial_\vartheta$ on $Q$ with $\vartheta = -\frac{1}{4t}u-v$ 
Then the \Cref{nonlineareq} becomes $X'=X^2-t$, which is a Riccati differential equation that is not solvable by any algebraic function over $k$ (see \cite[Example p.42-43]{Kap}). 

Let $K=k(\lambda)$ be a rational function field where $\lambda'=\lambda^2-t$ and $\lambda+\frac{1}{\theta}$ be another solution of $X'=X^2-t$ in $\ovl{K}$, i.e $\theta\in \ovl{K}$. Then it is easy to observe that $\theta'=-2\lambda\theta-1$.
Let $p(X)=X^m+a_{m-1}X^{m-1}+\dots+a_1X+a_0$ be the minimal polynomial of $\theta$ over $K$.  Consider a differential homomorphism $\phi:K[X]\to K[\theta]$ defined by $\phi(X)=\theta$. Then the kernel of $\phi$ is the prime ideal generated by $p(X)$. Since $p'(\theta)=0$, $p$ divides $p'$ and whereby $p(X)'=-2m\lambda p(X)$.
On comparing the coefficients of $X^{m-1}$, we obtain $a_{m-1}$ is a solution of the differential equation 
\begin{equation}\label{firstorder}
Y'=-2\lambda Y+c
\end{equation}
over $K$, when $c=m$.
 Let $m\neq 0$ then $a_{m-1}\neq0$. Write $a_{m-1}=\frac{f}{g}$ where $f,g$ are coprime polynomials in $k[\lambda]$ and $g$ is monic. Then $f'g-fg'=-2\lambda fg+mg^2$ implies $g$ divides $g'$.

Let $\gamma_1, \ldots, \gamma_r \in \ovl{k}$ and $m_1, \ldots, m_r \in \zz$ be such that $g = \prod_{i=1}^r (\lambda - \gamma_i)^{m_i} $.
Then 
$$g' = \prod_{i=1}^r (\lambda - \gamma_i)^{m_i-1} \cdot \left(\sum_{i=1}^{r} \Big((\lambda'-\gamma_i')\prod_{j\neq i} (\lambda - \gamma_j)\Big)\right).$$ 
Since $g$ divides $g'$, for $1\leq i\leq r$, we have $(\lambda-\gamma_i)$ divides $(\lambda'- \gamma_i') = \lambda^2-t-\gamma_i' $,
which implies that each $\gamma_i$ is a solution of $X' = X^2-t$. Therefore, each $m_i=0$ and $g=1$. 

Write $f=\sum_{i=0}^sf_i\lambda^i$, where $f_i\in k$, $f_s\neq0$. Then $f'=-2\lambda f+c$ implies $s=-2$, which is a contradiction. Therefore, $m=0$ and $\theta\in K$ is a solution of \Cref{firstorder} for $c=-1$, which is again not possible. Thus, any field extension of $k$ containing two distinct solutions of $X'=X^2-t$ must be of transcendence degree at least $2$. By \Cref{splittransnotone}, splitting fields of $((1,t)_k,d)$ have transcendence degree at least $3$ over $(k,\,')$.
\end{example}

\section{Splitting fields with transcendence degree $1$}{}\label{finitefields}
 In \Cref{splittransone}, we have seen that for derivations of the form $d =  d_{(u,v)}+ \partial_{au}$ where $a\in k$, the transcendence degree of splitting field is at most one. We will further analyze these derivations to have the transcendence degree equal to $1$.
\begin{proposition}\label{finitesplit}
Let $(Q,d)$  be a differential quaternion algebra where $d =  d_{(u,v)}+ \partial_{au}$ for $a\in k$. 
Then  $(Q,d)$ is split by a finite extension of $(k,\,')$, if and only if
$$a \in \left\{ \frac{\theta'}{n\xi\theta}\mid \theta \in k(\xi), n\in \nat \right\}.$$
\end{proposition}

\begin{proof}
By \Cref{splitequivalence}, $(Q,d)$ is split by a finite extension $(L,\,')$ of $(k,\,')$ if and only if 
$L$ contains a solution of the equation
\begin{equation}\label{lineareq}
 X' =  \left(a\xi + \frac{\beta'}{4\beta}\right) X,
 \end{equation}
that is, if and only if 
$L(\sqrt[4]{\beta})$ contains a solution of the equation $ X' =  a\xi X$.
Let $\lambda \in \ovl{k(\xi)}$ be a solution of \Cref{lineareq} and $ \theta = N_{k(\xi)(\lambda)/k(\xi)}(\lambda) $.
For $n=[k(\xi)(\lambda): k(\xi)]$, we have $\theta ' = n  a\xi \theta$.
\end{proof}

\begin{lemma}
For $\theta \in k(\xi)$, $\frac{\theta'}{\theta}\in \xi k$ if and only if  $N_{k(\xi)/k}(\theta) \in C_k$.
\end{lemma}

\begin{proof}
Let $\ovl{\theta}$ be the conjugate of $\theta$. 
Note that $\ovl{\theta}' = \ovl{\theta'}$. 
Then $\frac{\theta'}{\theta} \in \xi k$ if and only if 
$\frac{\theta'}{\theta}= - \ovl{\frac{\theta'}{\theta}}$,
 that is, $ (N_{k(\xi)/k}(\theta))' = \theta' \cdot \ovl{\theta} + \theta \cdot (\ovl{\theta})' =0$.
\end{proof}

\begin{remark}\label{constantnorm}
To describe the set of all $a \in k$ such that $(Q,d)$ is split by a finite extension of $(k, \,')$ we need to find the set $S = \{\theta \in k(\xi)\setminus k \mid N_{k(\xi)/k}(\theta) \in C_k \}$. 
Consider the non-empty subset $\tilde{S} = \{\theta  \in k(\xi) \setminus k\mid N_{k(\xi)/k}(\theta) \in \sq{C}_k \}$ of $S$. 
By Hilbert's Theorem 90 \cite[Example 2.3.4]{GilSza}, we have
$$\tilde{S} = C_k \cdot \left\{   \frac{\gamma}{\ovl{\gamma}} \mid   \gamma \in \mg {k(\xi)}  \right\}. $$

\end{remark}

\begin{example}\label{rationalfinitesplit}

Let $(k,\,')$ be the differential field where $k =\qq(t)$ and $t' = 1$.
Let $\alpha \in \qq[t]$ be an odd degree polynomial.
Then  
$S= \tilde S$ as in \Cref{constantnorm}.
For the derivation $d =  d_{(u,v)}+ \partial_{au} $ where $a\in k$, $(Q,d)$ is split by a finite extension of $(k,\,')$ if and only if 
$$a \in \left\{ \frac{\theta'}{n\xi\theta}\mid \theta \in \tilde S, n\in \nat \right\} .$$
\end{example}

\begin{theorem}\label{transdegone}
Let $(k,\,')$ be the differential field where $k =\qq(t)$ and $t' = 1$.
Let $\alpha \in \qq[t]$  be an odd degree polynomial.
Consider the quaternion algebra $Q = (\alpha, \beta)$ with  derivation $d=  d_{(u,v)} + \partial_{au}$ where $a =\frac{f}{g}\in \qq(t)$.
Then $(Q,d)$ is not split by any algebraic extension of $k$, if
\begin{enumerate}[$(a)$]
\item $deg(g)-deg(f)< \frac{deg(\alpha) +3}{2}$.
\item for a factor $h$ of $\alpha$, $h^2$ divides $g$.
\end{enumerate}
\end{theorem}

\begin{proof}
Suppose $(Q,d)$ is split by an algebraic extension of $k$.
Then, by \Cref{rationalfinitesplit}, there exists $\theta \in \tilde S$ such that $a = \frac{\theta'}{n\xi\theta}$. 
It follows from \Cref{constantnorm} that
$ \theta = c \frac{\gamma_0 + \xi \gamma_1}{\gamma_0 - \xi \gamma_1}$
where $c\in \qq$ and $\gamma_0, \gamma_1\in \qq[t]$ are coprime polynomials.
Observe that
$$
\frac{\theta'}{n\xi\theta} =  \frac{2\alpha( \gamma_0 \gamma_1' - \gamma_0' \gamma_1) + \alpha' \gamma_0\gamma_1}{ n\alpha(\gamma_0^2 -\alpha \gamma_1^2)} \in \qq(t).
$$
Set $p = 2\alpha( \gamma_0 \gamma_1' - \gamma_0' \gamma_1) + \alpha' \gamma_0\gamma_1$ and $q = n\alpha(\gamma_0^2 -\alpha \gamma_1^2)$.  

\begin{enumerate}[$(a)$]
\item Let $m =deg(\alpha), m_0 = deg( \gamma_0) $ and $m_1 = deg( \gamma_1)$.
We have $deg(p) \leq m + m_0+ m_1 -1$ and $deg(q) = m + \max( 2 m_0, 2m_1 + m)$ 
and calculating further, we get that 
$deg(q) - deg(p)  \geq  \frac{m +3}{2} $.

\item Let $h \in \qq[t]$ be an irreducible factor of $\alpha$.
 Write $\alpha = h\tilde{\alpha}$ with $ \tilde{\alpha} \in \qq[t]$.
If $h \nmid \gamma_0$ then $h^2 \nmid q$, whereby $h^2 \nmid g $.
If $h \mid \gamma_0$ then $h^2|q$, $h^3 \nmid q$ and $h|p$, whereby $h^2 \nmid g $. 
\end{enumerate}
\end{proof}

\section{Standard derivations}\label{SD}

Let $(k,\,')$ be a differential field and let $(Q,d)$ be  a differential quaternion algebra over $(k,\,')$. 
Recall that $d$ is a standard derivation on $Q$ if there exist $u,v \in Q^0$ such that $u^2,v^2\in \mg k, uv = -vu, d(u)\in k(u)$ and $d(v)\in k(v)$. 
For a differential field extension $(L,\,')$ over $(k,\,')$, we say that \textbf{$d$ becomes standard over $L$} if $d^{\ast} = d\otimes \,'$ is standard derivation on $Q\otimes_k L$.
A quaternion algebra with standard derivation is split by a finite field extension of degree at most $8$ (see \cite[Theorem 4.1]{AKVRS}), but the following example shows that the converse need not be true.

\begin{example}\label{examplealgebraic}
Let $(k,\,')$ be the differential field where $k =\qq(t)$ and $t' = 1$. 
Let  $Q=(1,t)_k$ with $u,v\in Q$ such that $u^2=1$, $v^2=t$ and $vu=-uv$.  
Consider the derivation $d=d_{(u,v)}+\partial_{-\frac{1}{8t}u}$ on $Q$. 
Then Equations (\ref{nonlineareq}) and (\ref{linearriccati}) will become $X' = \frac{1}{4t}X$ and $Y' = -\frac{1}{8t}Y$, respectively, and it follows  
from \Cref{splitequivalence} that a field extension $(L,\,')$ splits $(Q,d)$ if and only if $L$ contains solutions of these equations, that is, if and only if $L$ contains $k(\sqrt[8]{t})$. 
Thus, $[L:k]\geq 8$.

We now show that $d$ is not a standard derivation on $Q$.
Let $\tilde u = a_1 u +a_2v +a_3uv$ such that ${\tilde u}^2 \in \mg k$ and $d(\tilde u) \in k\tilde u$.
Then $d(\tilde u) \in k\tilde u$ implies that $ d(\tilde u) =\frac{N(\tilde u)'}{2N(\tilde u)} \tilde{u} $ and we get the following equations:
\begin{align}
\frac{N(\tilde u)'}{2N(\tilde u)} a_1  & = a_1' \label{standardeq1}\\  
 \frac{N(\tilde u)'}{2N(\tilde u)} a_2 &= a_2' + \frac{1}{2t}a_2  +\frac{a_3}{4t} \label{standardeq2}\\
\frac{N(\tilde u)'}{2N(\tilde u)} a_3   &=  a_3' + \frac{1}{2t}a_3 +\frac{a_2}{4t} \label{standardeq3}
\end{align}
Note that $a_2 =0$ if and only if $a_3 =0$.  
Suppose that $a_2,a_3 \in \mg k$.
From Equations (\ref{standardeq2}) and (\ref{standardeq3}), we obtain that $\frac{a_2}{a_3}$ is a solution of the Riccati Equation $ X' = \frac{1}{4t}(X^2 -1 )$. 
Note that $\pm 1$ are solutions of this equation and using separation of variables method, we observe that the general solution other than $\pm 1$ is $X = \frac{1+c\sqrt{t}}{1 -c\sqrt{t}}$ with $c' =0$.
Since $a_2, a_3 \in k$, we get that $\frac{a_2}{a_3} = \pm 1$.
Thus $N(\tilde u) =a_1^2 \neq 0$, whereby $a_1 \in \mg k$.
If $a_2 =a_3$, it follows from Equations (\ref{standardeq1}) and (\ref{standardeq2}) that  $\frac{a_1}{a_2}$ is a solution of the linear equation $Y' = \frac{3}{4t}Y$ and hence  $\frac{a_1}{a_2}$ is a constant multiple of $t^{\frac{3}{4}}$, which is a contradiction.
Similarly, if $a_2 = - a_3$, we get that  $\frac{a_1}{a_2}$ is a solution of $Y' = \frac{1}{4t}Y$, whereby  $\frac{a_1}{a_2}$ is a constant multiple of $t^{\frac{1}{4}}$, which is again a contradiction.
Thus $a_2 =a_3 =0$, and  $\tilde u \in ku$. 

Thus any $\tilde u \in Q^0$ such that $d(\tilde u) \in k\tilde u$ is an element of  $ku$.
Since any  pair of such elements cannot generate $Q$ over $k$, 
 $d$ is not a standard derivation on $Q$.
\end{example}

Now, we show that a  derivation on a quaternion algebra becomes standard over its differential splitting  fields.
\begin{proposition}\label{splittingimpliesstandard}
	Let $(Q,d)$ be a differential quaternion algebra over $(k,\,')$ and let $(L,\,')$ be a differential splitting field. Then $d$ becomes standard over $L$.	
\end{proposition}
\begin{proof}
	 Since $(Q\otimes_k L,d^*)\simeq (M_2(L), d_P) $ for a  trace zero matrix $P$, we have $(M_2(L), d_P)\simeq (M_2(L),\,')$. 
By \Cref{JM_split}, there exists $F\in GL_2(L)$ such that $F'=PF$. 
By \Cref{splitequivalence},  $L$ contains a solution of \Cref{linearriccati}, say  $\theta$. 
Set $$U=F\left(\begin{matrix}
		\theta&0\\
		0&-\theta
	\end{matrix}\right)F^{-1} \mbox{and  } V=F\left(\begin{matrix}
		0&\theta\\
		\theta&0
	\end{matrix}\right)F^{-1}.$$ Then $U^2=V^2=\theta^2$, $VU=-UV$, $d_P(U)=\frac{\theta'}{\theta} U$ and $d_P(V)=\frac{\theta'}{\theta} V$. Let $\tilde u$ and $\tilde v$ be pre-images of $U$ and $V$ respectively, in $Q\otimes_k L$.  Then $\tilde u$ and $\tilde v$ are generators of $Q\otimes_k L$ with $d^*(\tilde{u})=\frac{\theta'}{\theta} \tilde{u}$ and  $d^*(\tilde{v})=\frac{\theta'}{\theta} \tilde{v}$. Thus, $(Q\otimes_k L,d^*)=(Q\otimes_k L,d_{(\tilde u,\tilde v)})$.	
\end{proof}

\begin{theorem}\label{standardequivalance}
	Let $(Q,d)$ be a differential quaternion algebra over $(k,\,')$.
	Then $d$ becomes standard over a finite extension of $k$ if and only if $(Q,d)$ is split by a finite extension of~$k$.	
\end{theorem}
\begin{proof}
	Suppose	 $d$ becomes standard over a finite extension of $k$, then using \Cref{Quatex}(\ref{Quatex2}) we get that $(Q,d)$ is split by a finite extension of $k$.
	
	Conversely, let $(Q,d)$ be split by a finite extension $(L,\,')$ of $(k,\,')$. Then by \Cref{splittingimpliesstandard}, $d$ becomes standard over $L$.
\end{proof}

\begin{remark}
	The class of differential quaternion algebras over $\qq(t)$ defined in \Cref{transdegone} provides examples of derivations that do not become standard over any finite extension of $\qq(t)$.
\end{remark}


%
%
\subsection*{Acknowledgement}
The authors wish to thank  Amit Kulshrestha and Varadharaj R.~Srinivasan for interesting discussions on this topic and raising \Cref{standardquestion}.

\bibliographystyle{amsalpha}

\begin{thebibliography}{}

\bibitem{Ami}
	S.~A.~Amitsur, \emph{Extensions of derivations to central simple algebras}, 
	Commun.~Algebra, vol.~\textbf{10} issue~8 (1982),  797--803.


	\bibitem{GilSza}
	P.~Gille and T.~Szamuely, 
	\emph{Central simple algebras and Galois cohomology},
	Cambridge Studies in Advanced Mathematics, vol.~{\bf 101}, 
	Cambridge University Press, 2006.
		
\bibitem{GKS}
	P.~Gupta, Y.~Kaur, A.~Singh, 
	\emph{Splitting fields of differential symbol algebras}, J.~Pure Appl.~Algebra vol.~{\bf 227} Issue: 5 (2023), 1-16.	

\bibitem{Ho} G.~Hochschild, \emph{Restricted Lie algebras and simple associative algebras of characteristic $p$} Trans.~Am.~Math.~Soc. \textbf{80} (1955), 135--147.
	

\bibitem{Hoe}
	K.~Hoechsmann, 
	\emph{Simple algebras and derivations},
	Trans.~Amer.~Math.~Soc. \textbf{108} no.~1 (1963), 1--12.

\bibitem{LJARM} L.~Juan and A.~R.~Magid, 
	\emph{Differential central simple algebras and Picard-Vessiot representations}, 
	Proc.~Amer.~Math.~Soc. vol.~{\bf 136}, no.~6 (2008), 1911--1918.
         


\bibitem{Kap}I.~Kaplansky, {\em An introduction to differential algebra},  Actualit\'es Sci. Ind., 1251, Publ. Inst. Math. Univ. Nancago, vol. \textbf{5}, Hermann, Paris, 1957.  	         

\bibitem{AKVRS} 
	A.~Kulshrestha and V.~R.~Srinivasan, 
	\emph{Quaternion algebras with derivations}, 
	J.~Pure Appl.~Algebra vol.~{\bf 226} issue 2 (2022), 1--14.
	

\bibitem{Ros}
	M.~Rosenlicht, \emph{Integration in finite terms}, 
	Am.~Math.~Mon. vol.~{\bf 79} issue 9 (1972),  963--972.

\end{thebibliography}

\end{document}